\documentclass{article}
\usepackage{amsthm, amsmath, amssymb,url}

\usepackage{xcolor}

\newtheorem{theorem}{Theorem}

\newtheorem{proposition}[theorem]{Proposition}

\newtheorem{example}[theorem]{Example}
\newtheorem{remark}[theorem]{Remark}
\newtheorem{defn}[theorem]{Definition}
\newtheorem{algorithm}[theorem]{Algorithm}
\def\C{\mathcal{C}}
\def\U{\mathcal{U}}
\def\K{\mathbb{K}}
\def\F{\mathbb{F}}

\title{Computing Hypercircles by Moving Hyperplanes}
\author{Luis Felipe Tabera}

\begin{document}
\maketitle

\begin{abstract}
Let $\mathbb{K}$ be a field of characteristic zero, $\alpha$ algebraic of degree
$n$ over $\mathbb{K}$. Given a proper parametrization $\psi$ of a rational curve
$\C$, we present a new algorithm to compute the hypercircle associated to the
parametrization $\psi$. As a consequence, we can decide if $\C$ is defined over
$\mathbb{K}$ and, if not, to compute the minimum field of definition of
$\mathcal{C}$ containing $\K$. The algorithm exploits the conjugate curves of
$\C$ but avoids computation in the normal closure of $\mathbb{K}(\alpha)$ over
$\mathbb{K}$.
\end{abstract}

\section{Introduction}
Let $\mathbb{K}(\alpha)$ be a computable characteristic zero field with
factorization such that $\K$ is finitely generated over $\mathbb{Q}$ as a field
and $\alpha$ is of degree $n$ over $\K$.

Let $\psi(t) = (\psi_1(t), \ldots, \psi_m(t))$ be a proper parametrization of a
rational spatial curve $\C$, where $\psi_i\in \mathbb{K}(\alpha)(t)$, $1\leq
i\leq m$. The reparametrization problem ask for methods to decide in $\C$ is
defined or parametrizable over $\mathbb{K}$ and, if possible, compute a
parametrization of $\C$ over $\K$.

In \cite{ARS99}, the authors proposed a construction to solve this problem
introducing a family of curves called hypercircles and avoiding any
implicitization technique. Starting from the parametrization $\psi$, they
construct an analog to Weil descente variety to compute a curve $\mathcal{U}$
called the \emph{witness variety} or the \emph{parametric variety of Weil}. This
curve exists if and only if $\C$ is defined over $\K$ and we can obtain a
parametrization of $\C$ with coefficients in $\K$ easily from a parametrization
of $\mathcal{U}$ with coefficients in $\K$. Efficient algorithms to compute a
parametrization of $\mathcal{U}$ with coefficients in $\K$ are studied in
\cite{Algorithmic_detection_of_HC}, provided we are able to find a point in
$\mathcal{U}$ with coefficients over $\mathbb{K}$.

The definition of $\mathcal{U}$ is done under a parametric version of Weil's
descente method. In the proper parametrization $\psi=(\psi_1, \ldots, \psi_m)$,
$\psi_i \in \mathbb{K}(\alpha)(t)$ with coefficients in $\K(\alpha)$, we
substitute $t=\sum_{i=0}^{n-1} \alpha^it_i$, where $t_0, \ldots, t_{n-1}$ (where
$n$ is the degree of $\alpha$ over $\mathbb{K}$). 

We can rewrite:
\[\psi_j\left(\sum_{i=0}^{n-1} \alpha^it_i\right)= \sum_{i=0}^{n-1} \alpha^i
\lambda_{ij}(t_0,\ldots,t_{n-1}), \lambda_{ij}= \frac{F_{ij}}{D} \in
\K(t_0,\ldots, t_{n-1})\]
In this context we have the following definition:

\begin{defn}\label{defn:witness}
The parametric variety of Weil $\mathcal{Z}$ of the parametrization $\psi$ is he
Zariski closure of
\[\{F_{ij}=0\ |\ 1\leq i\leq n-1,\ 1\leq j\leq N\} \setminus \{D=0\}\subseteq
\F^n.\]
\end{defn}

Much is known about $\mathcal{Z}$, it is always a set of dimension $0$ or $1$.
It is of dimension one exactly in the case that $\mathcal{C}$ is defined over
$\K$ (See \cite{ARS99}, \cite{Ultraquadrics}). In this case, $\mathcal{Z}$
contains exactly one component of dimension $1$ that is the searched curve
$\mathcal{U}$. 

The computation of the curve $\mathcal{U}$ from its definition is unfeasible
except for toy examples. The curve $\mathcal{U}$ is defined as the unique one
dimensional component of a the difference of two varieties
$\mathcal{A}-\mathcal{B}$. This already is a hard enough problem to look for
alternatives, but this method also uses huge polynomials. If $\psi_i(t) =
n_i(t)/d(t)$ and $d=d(\alpha, t)\in \mathbb{K}[\alpha,t]$. Let $M(x)$ be the
minimal polynomial of $\alpha$ over $\mathbb{K}$. In the generic case, the
denominator $D$ is $D=Res_z(d(z,\sum_{i=0}^{n-1}z^it_i), M(z))$ which is
typically a dense polynomial of degree $dn$ in $n$ variables. Hence, the number
of terms of the polynomial $D$ alone is not polynomially bounded in $n$.

The aim of the article is to present an algorithm to compute the variety
$\mathcal{U}$ that is polynomial in $d$ and $n$ and, if $\C$ is not defined over
$\K$, to compute the smallest field $\mathbb{L}$, $\K\subseteq
\mathbb{L}\subseteq \K(\alpha)$ that defines $\C$. The article is structured as
follows. First we introduce in Section~\ref{sec:syntetic_construction} the
geometric construction that will allow us to derive an efficient 
algorithm. Then, we show
in Section~\ref{sec:efficient_computation} how to compute efficiently some steps
of the algorithm. Last, in Section~\ref{sec:running}, we study the complexity of
the algorithm and some running times comparing with other approaches.

\section{Synthetic construction of
Hypercircles}\label{sec:syntetic_construction}

The problem of parametrizing $\C$ over $\K$ can be translated to the problem of
parametrizing $\mathcal{U}$. In the case that $\C$ can be parametrized over
$\mathbb{K}$, then $\mathcal{U}$ is a very special curve called hypercircle.

\begin{defn}\label{defn:hypercircle}
Let $\frac{at+b}{ct+d}\in \K(\alpha)(t)$ represent an isomorphism of $\F(t)$,
$a,b,c,d\in \K(\alpha)$, $ad-bc\neq 0$. Write
\[\frac{at+b}{ct+d}=\lambda_0(t)+\alpha \lambda_1(t)+ \cdots+
\alpha^{n-1}\lambda_{n-1}(t)\] where $\lambda_i(t)\in \K(t)$. The
\textit{hypercircle} associated to $\frac{at+b}{ct+d}$ for the extension
$\K\subseteq \K(\alpha)$ is the parametric curve in $\F^n$ given by the
parametrization $(\lambda_0,\ldots, \lambda_{n-1})$.
\end{defn}

If $\C$ cannot be parametrized over $\K$ and $\K$ is small enough (that means
that it is finitely generated over $\mathbb{Q}$ as a field, that we can always
assume without loss of generality), then there always exists an element $\beta$
algebraic of degree 2 over $\K$ such that $[\K(\beta, \alpha):\K(\alpha)] = n$
and $\C$ can be parametrized over $\K(\beta)$, see
\cite{Affine_reparametrization} for the details. In this situation $\U$ is a
hypercircle for the extension $\mathbb{K}(\beta) \subseteq
\mathbb{K}(\alpha,\beta)$. That is, there is an associated unit
$\frac{at+b}{ct+d}$, but with $a,b,c,d\in \mathbb{K}(\beta)$.

Thus, the curve $\mathcal{U}$ is always a hypercircle for certain algebraic
extension. So all the geometric properties of hypercircles studied in
\cite{Generalizing-circles} hold for $\mathcal{U}$ except, maybe, the existence
of a point in $\mathcal{U}\cap \mathbb{K}^n$. We will exploit the geometric
properties of hypercircles to derive our algorithm. We start with the fact that
$\mathcal{U}$ is always a rational normal curve in $\mathbb{F}^n$ defined over
$\mathbb{K}$ (See \cite{Generalizing-circles}) and the synthetic construction of
rational normal curves as presented in \cite{harris}. 

Let us recall the construction of conics by a pair of pencil of lines. Let
$\mathfrak{L}(t)$ and $\mathfrak{F}(t)$ be two different pencils of lines in the
plane with two different base points $l_0 \neq f_0$ and let $C$ be a conic
passing trough $l_0$ and $f_0$. Then, $C$ induces an isomorphism  $u :
\mathfrak{L}(t)\rightarrow \mathfrak{F}(t)$ given by extending the map
$u(\mathfrak{L}(t_0)) = \mathfrak{F}(s_0)$ if
\[\mathfrak{L}(t_0) \cap C -\{l_0\} = \mathfrak{F}(s_0)\cap C-\{f_0\}.\]

Conversely, an isomorphism $u$ between $\mathfrak{L}(t)$ and $\mathfrak{F}(t)$
defines a line or a conic passing through the base points. There is a proper
parametrization of this curve given by $t\mapsto \mathfrak{L}(t) \cap
\mathfrak{F}(u(t))$.

\begin{example}\label{ex:pencil_of_lines}
Let $\mathcal{C}=x^2+y^2-1$ be the unit circle. And take the pencils of lines
that passes through the points at infinity of the circle $[1:i:0], [1:-i:0]$.
$\mathfrak{L}(t) = \{x+iy = t\}$, $\mathfrak{F}(t) = \{x-iy = t\}$. In this case
$\C \cap \mathfrak{L}(t) = (\frac{t^{2} + 1}{2 t}, \frac{-i t^{2} + i}{2 t})$
and $\C \cap \mathfrak{F}(t) = (\frac{t^{2} + 1}{2 t}, \frac{i t^{2} - i}{2
t})$. In this case, the isomorphism between the pencils is given by $u(t) =
1/t$. Now, let us take the isomorphism $u(t)=(t+i)/t$. Then, the conic defined
by $u$ from the two pencils of lines is $x^2 + y^2 - x -iy -i$. Which is a conic
passing through the base points, although not defined over $\mathbb{Q}$. 
\end{example}

More generally, the same geometric construction applies to rational normal
curves of degree $n>2$ in $\mathbb{F}^n$ as explained in \cite{harris}. We only
show the special case of this construction that is relevant for hypercircles. If
$\mathcal{U}$ is a hypercircle, it is known that $\mathcal{U}$ can be
parametrized by the pencil of hyperplanes $\mathfrak{L}_0 = \{\sum_{i=0}^{n-1}
\alpha^i x_i = t\}$ \cite{Generalizing-circles}. This pencil of hyperplanes
yield to a proper parametrization $\phi=(\phi_0(t),\ldots, \phi_{n-1}(t))$ of
the hypercircle with coefficients in $\K(\alpha)$ that is called the
\emph{standard parametrization} of the hypercircle and has been studied with
detail in \cite{Algorithmic_detection_of_HC}. Since the hypercircle is always a
curve defined over $\K$, it is invariant under conjugation and it can also be
parametrized by the conjugate pencil of hyperplanes.

Let us fix some notation. Let $\alpha=\alpha_0, \alpha_1, \ldots, \alpha_{n-1}$,
the conjugates of $\alpha$ over $\K$ in $\F$. Let $\sigma_i$, $0\leq i\leq
{n-1}$ be $\mathbb{K}$-automorphisms of $\mathbb{F}$ such that
$\sigma_i(\alpha)=\alpha_i$ and $\sigma_0=Id$. If we have a rational function
$f(t)\in\mathbb{K}(\alpha)(x_1, \ldots, x_r)$, we denote by $f^{\sigma_j}=
\sigma_j(f)\in \mathbb{K}(\alpha_j)(x_1, \ldots, x_r)$ that results applying
$\sigma_j$ to the coefficients of $f$. If $\C$ is the original curve, then we
denote by $\C^\sigma$ the conjugate curve $C^\sigma = \{\sigma(x)| x \in C\}$,
where $\sigma(x)$ is applied component-wise. $\C^\sigma$ is clearly a rational
curve with proper parametrization $\psi^\sigma$.

It is known \cite{Ultraquadrics} that $\C$ is defined over $\K$ if and only if
$\C=\C^{\sigma_i}$ $1\leq i\leq n-1$ if and only if $\psi^{\sigma_i}$
parametrizes $\C$, $1\leq i\leq n-1$.

The conjugate pencil of hyperplanes $\mathfrak{L}_j(t) =
\{ \sum_{i=0}^{n-1} \alpha_j^i x_i = t \}$, $1\leq j\leq n-1$ also parametrizes
$\mathcal{U}$, yielding the conjugate parametrization
$\phi^{\sigma_j}(t)=\sigma_j(\phi(t))$.

The hypercircle then induces an isomorphism $u_j(t)$ between $\mathfrak{L}_0(t)$
and $\mathfrak{L}_j(t)$ given by $(\mathfrak{L}_0(t_0) \cap \U) - H =
(\mathfrak{L}_j(u_j(t_0)) \cap \U)-H$ for all but finitely many parameters
$t_0$, where $H$ is the hyperplane at infinity of $\mathbb{P}(\F)^n$. So
$\phi(t) = \phi^{\sigma_j}(u_j(t))$, from which $u_j(t) =
(\phi^{\sigma_j})^{-1}\circ \phi$. But, by construction,
$(\phi^{\sigma_j})^{-1}(x_0,\ldots, x_{n-1})=\sum_{i=0}^{n-1}\alpha_j^i x_i$ and
$u_j(t) = \sum_{i=0}^{n-1} \alpha_j^i \phi_i(t)$. Conversely, a set of
isomorphisms $u_j: \mathfrak{L}_0(t) \rightarrow \mathfrak{L}_j(t)$, $0\leq
j\leq n-1$, $u_0(t)=t$, defines a rational normal curve given by $t\rightarrow
\bigcap_{i=0}^{n-1} \mathfrak{L}_j(u_j(t))$. So, we can recover the standard
parametrization of the hypercircle if we know the isomorphisms $u_j$, $0\leq
j\leq n-1$, where $u_0(t)=t$. The standard parametrization $\phi$ is the unique
solution of the Vandermonde linear system of equations:
\begin{equation}\label{eq:Vandermonde_standard}\begin{pmatrix}
 1 & \alpha & \ldots & \alpha^{d-1}\\
 1 & \alpha_2 & \ldots & \alpha_2^{d-1}\\
 \multicolumn{4}{c}{\ldots}\\
 1 & \alpha_n & \ldots & \alpha_n^{d-1}\\
 \end{pmatrix}
 \begin{pmatrix}
 \phi_0 \\ \phi_1\\ \ldots \\ \phi_{n-1}
 \end{pmatrix}
 =
 \begin{pmatrix}
 u_0(t) = t \\ u_1(t)\\ \ldots \\ u_{n-1}(t)
 \end{pmatrix}
\end{equation}
with coefficients on the normal closure of $\alpha$ over $\mathbb{K}$.

As in the planar case, if the automorphisms are generic enough, the curve
$\mathcal{U}$ will be of degree $n$. In this case we say that $\mathcal{U}$ is a
\emph{primitive} hypercircle. There may be cases in which the curve
$\mathcal{U}$ is of degree less than $n$. If this is the case, the degree of
$\mathcal{U}$ must be a divisor of $n$ and is related with the field of
definition of the place of $\C$ corresponding to $\psi(t=\infty)$, as showed in
\cite{Affine_reparametrization}.

The good news is that we can compute easily the automorphisms $u_j(t)$ from the
parametrization $\psi(t)$ alone.

\begin{theorem}\label{teo:howto_compute_u}
Let $\psi(t)\in \mathbb{\K}(\alpha)(t)^m$ be a proper parametrization of $\C$
and assume that $\C$ is defined over $\K$. Let $\phi(t)$ be the standard
parametrization of the associated hypercircle $\mathcal{U}$. Let $\sigma_i$ be a
$\mathbb{K}$-automorphism of $\mathbb{F}$. Let $\phi^{\sigma_i} =
\sigma_i(\phi)$, $\psi^{\sigma_i} = \sigma_i(\psi)$ be the conjugate
parametrizations and $u_{\sigma_i}(t) = (\phi^{\sigma_i})^{-1} \circ \phi$ be
the conjugation isomorphism induced by $\mathcal{U}$ in the pencil of
hyperplanes $\mathfrak{L}_0$ and $\mathfrak{L}_{\sigma_i}$. Then $u_{\sigma_i} =
(\psi^{\sigma_i})^{-1}\circ \psi$.
\end{theorem}
\begin{proof}
We identify $\C$ with the diagonal curve $\Delta$ in the variety $\C\times
\C^{\sigma_1} \times \ldots \times \C^{\sigma_{n-1}}$, $\Delta = \{(x,\ldots,
x)| x\in \C\}$ the hypercircle $\mathcal{U}$ is a curve such that the map 

\[\begin{matrix}\U &\rightarrow &\C\times \C^{\sigma_1}\times \ldots \times
\C^{\sigma_{n-1}} \\
(x_0,\ldots, x_{n-1}) & \mapsto &
\displaystyle{\left(\psi(\sum_{j=0}^{n-1}x_j\alpha^j),
\psi^{\sigma_1}(\sum_{j=0}^{n-1}x_j \alpha_1^j), \ldots,
\psi^{\sigma_{n-1}}(\sum_{j=0}^{n-1}x_j\alpha_{n-1}^j)\right)}
\end{matrix}\]

Is a birational map between $\U$ and $\Delta = \C$. See \cite{Ultraquadrics}
for the details. This means that $\psi(\sum_{j=0}^{n-1}x_j\alpha^j) =
\psi^{\sigma_i}(\sum_{j=0}^{n-1}x_j\alpha_i^j)$ for the points of the
hypercircle. If we plug the standard parametrization of the hypercircle in this
equality, we get that
\[\psi(t) = \psi(\sum_{j=0}^{n-1} \phi_j\alpha^j) =
\psi^{\sigma_i}(\sum_{j=0}^{n-1}\phi_j\alpha_i^j)=
\psi^{\sigma_i}(u_{\sigma_i}(t))\]
From which $u_{\sigma_i} = (\psi^{\sigma_i})^{-1}\circ \psi$.
\end{proof}

Hence the isomorphism $u_j$ induced by the hypercircle in the pencil of
hyperplanes $\mathfrak{L}(t)$ and $\mathfrak{L}_j(t)$ is the change of variables
needed to transform the conjugate parametrization $\psi^{\sigma_j}(t)$ into
$\psi(t)$. We can compute $u_j$ using $\gcd$.

\begin{theorem}\label{teo:automorphism_by_gcd}
Let $\psi_i(t) = n_i(t) / d_i(t)$ and
$\psi_i^{\sigma_j}(t)=n_i^{\sigma_j}(t)/d_i^{\sigma_j}(t)$ be the numerators and
denominators of $\psi$ and $\psi^j$. Then, if $\C$ is defined over $\K$, the
numerator of $s - u_j(t)$ is a polynomial of degree 1 in $t$ and in $s$ that is
the common factor of the set of polynomials
\[B_{\sigma_j}=\{n_i(t)\cdot d_i^{\sigma_j}(s) - n_i^{\sigma_j}(s)\cdot d_i(t),
1\leq i\leq m\}.\]
On the other hand, if $\C$ is not defined over $\K$, there is an index $1\leq
j\leq n-1$ such that $\gcd(B_{\sigma_j})=1$.
\end{theorem}
\begin{proof}
This result follows directly from the geometric interpretation. First, assume
that $\C$ is defined over $\K$. It is clear that the numerator of $s-u_j(t)$ is
a common factor of the set $B_{\sigma_j}$. Let $f(t,s)$ be the $\gcd$ of
$B_{\sigma_j}$ and let $p=\psi(t_0)\in \C$ where $t_0$ is a generic evaluation
of $t$. The roots of $f(t_0,s)$ are solutions of the system of equations
$\psi_i^{\sigma_j}(s)=p_i$. But, since $\psi^{\sigma_j}$ is birational, for all
but finitely many $t_0$ there is only one solution, $(\psi^{\sigma_j})^{-1}(p)$.
Hence, the degree of $f$ with respect to $s$ is one. By symmetry, the degree of
$f$ with respect to $t$ is also one. It follows that $f$ must be the numerator
of $s-u(t)$.

Now, assume that $\C$ is not defined over $\K$. Then, there is an index $j$ such
that $\C\neq \C^{\sigma_j}$. In this situation, for all but finitely many
evaluations $t=t_0$, the system of equations $\psi^{\sigma_j}(s) = \psi(t_0)$
has no solution. It follows that $\gcd(B_{\sigma_j})=1$.
\end{proof}

So, we can compute $\K$-definability and the standard parametrization of the
hypercircle $\mathcal{U}$ by the following method:
\begin{itemize}
\item For each conjugate $\alpha_j$, Compute $a(t)+sb(t)$, the gcd of
$n_i(t)\cdot d_i^{\sigma_j}(s) - n_i^{\sigma_j}(s)\cdot d_i(t), 1\leq i\leq m$.
If one of the $\gcd$ is one, then the curve is not defined over $\K$ and we are
done.
\item Set $u_{j} = -a(t)/b(t)$.
\item Solve the linear system of equations (\ref{eq:Vandermonde_standard}) whose
coefficients are rational functions in $t$ with coefficients in the normal
closure of $\mathbb{K}(\alpha)$.
\end{itemize}

However, computing these bivariate $\gcd$ are expensive and, moreover, in the
worst case, we will have to solve a linear set of equations with coefficients in
an extension of $\mathbb{K}$ of degree $n!$. Next section address the problem of
how to perform this algorithm efficiently.

\section{Efficient Computation of the
Hypercircle}\label{sec:efficient_computation}

We have shown how to compute $u_\sigma(t)$ by computing the $\gcd$ of the
polynomials in $B_{\sigma}$. We already now that, if $\C$ is $\K$-definable, the
$\gcd$ has degree 1 in $t$ and $s$, so the best suited algorithms for computing
the $\gcd$ seem to be interpolation algorithms. Since we are only interested in
$u_\sigma$ and this linear fraction is an automorphism of $\mathbb{P}^1(\F)$, we
only need to know the image of three points $t_0, t_1, t_2$ under $u_\sigma$.
From Theorem~\ref{teo:automorphism_by_gcd}, for almost all $t_i$,
$u_{\sigma}(t_i)=(s_i)$ if and only if $\psi(t_i) = \psi^\sigma(s_i)$. Hence,
each $s_i$ is the common root of the polynomials:

\[\psi(t_i)\cdot d^{\sigma}_j(s) - n^{\sigma}_j(s), 1\leq i\leq m\]

$s_i$ that can be computed by means of $\gcd$ of univariate polynomials in
$\mathbb{K}(\alpha, \sigma(\alpha))$.

If $\C$ is defined over $\K$ then only finitely many parameters $t_k$ will fail
to provide a valid $s_k$. Essentially the parameters $t_k$ can fail if
$\psi(t_k)$ is a singular point of the curve or if it cannot be attained by a
finite parameter $s$ by the parametrization $\psi^\sigma$.

On the other hand, if $\C$ is not defined over $K$, then there is an
automorphism $\sigma$ such that $\C\neq \C^\sigma$. For this permutation, there
are only finitely many parameters $t_k$ such that $\psi(t_k) \in \C \cap
\C^\sigma$. Hence, if we want to follow this approach and do not depend on
probabilistic algorithms that may fail or give wrong answers, we need bounds to
detect that the curve is defined over $\K$ or not.

\begin{theorem}\label{teo:number_of_univariate_gcd_for_hypercircle}
Let $\mathcal{C} \subseteq \mathbb{F}^m$ be a rational curve of degree $d$ given
by a parametrization $\psi\in (\mathbb{K}(\alpha))^m$. Let $\alpha_i$ be any
conjugate of $\alpha$ over $\mathbb{K}$. Take $t_1, \ldots, t_{k}\in \F$
parameters then:
\begin{itemize}
\item If $\mathcal{C}$ is definable over $\mathbb{K}$, then we can compute
$u_i$ from three correct solutions of the system of equations
$\psi^\sigma(s)=t_k$.
\item If $\mathcal{C}$ is defined over $\mathbb{K}$, then at most $d^2-2d+n+1$
parameters can fail to give a correct answer.
\item If $\mathcal{C}$ is not defined over $\mathbb{K}$, then at most $d^2$
parameters $t_k$ will give a fake answer $s_k$.
\end{itemize}
\end{theorem}
\begin{proof}
We have to compute the inverse of the point $\psi(t_j)$ under the
parametrization $\psi^{\sigma_i}(s)$. For each $t_k$, this computation is done
using univariate $\gcd$. If we want to restrict to affine points, we have to
eliminate $d$ potential parameters of the denominator of $\psi$. Then, for an
affine point $\psi(t_j)$, there can only be one point that is not attained by a
finite parameter of $\psi^{\sigma_i}$. Since we have $n-1$ possible conjugates,
then there may be $n-1$ points that are not attained by a finite parameter in
one of the conjugate parmetrizations. So, if we get two different parameters
$t_j$ such that $\psi(t_j)$ is well defined but that $\psi^{\sigma_i}(s)=t_j$
have no solution (the corresponding $\gcd$ is $1$), then the curve is not
defined over $\mathbb{K}$. Now, it may happen that the $\gcd$ is of degree $>
1$. This can only happen if the point is singular in $\C$. Since $\mathcal{C}$
is of genus $0$ and degree $d$, it can have at most $(d-1)(d-2)/2$
singularities. The number of different parameters whose image is a singularity
is maximal if every singularity is ordinary. We have to maximize
\[\sum_{p\in sing(C)}mult_p(C)\]
subject to
\[\sum_{p\in sing(C)}mult_p(C)(mult_P(C)-1) = (d-1)(d-2)\]
See \cite{SWP} Theorem 2.60 for details.
But clearly, for any singular point $mult_p(C) \leq mult_p(C)(mult_P(C)-1)$ So
$\sum_{p\in sing(C)}mult_p(C)\leq (d-1)(d-2)$ and the equality is attained if
every singularity is an ordinary double point.

Thus, the maximal number of parameters that cannot be used to compute $u_i$ is
bounded by $d$ parameters corresponding to the points at infinity plus $n-1$
points that might not be attained by a finite parameter in a conjugate
parametrization $\psi^\sigma$ plus $(d-1)(d-2)$ parameters whose image are
singular points. This gives the bound $d^2-2d+n+1$.

Suppose now that $\mathcal{C}$ is not defined over $\K$. Let $\sigma_i$ be such
that $\C\neq \C^{\sigma_i}$. A parameter $t_0$ gives a fake answer for computing
$u_i$ if $\psi(t_0)$ is smooth in $C^{\sigma_i}$ and is attained by a unique
parameter $s_0$ by $\psi^{\sigma_i}$. But, by Bezout, $\C\cap \C^\sigma$
contains at most $d^2$ different points. So, there can be at most $d^2$ such bad
parameters.
\end{proof}

\begin{remark}\label{rem:cases_of_bad_good_parameters} In order to check that a
parameter $t_k$ is a good parameter or not we can do the following:

\begin{itemize}
\item If $t_k$ is a root of the denominator of $\psi$, then $t_k$ is a bad
parameter.
\item If $\gcd(\psi(t_k)\cdot d_i^{\sigma}(s) - n_i^{\sigma}(s), 1\leq i\leq
m)=1$, then it is a bad parameter. It is a point that is not attained by the
parametrization $\psi^{\sigma}$. If $\C$ is defined over $\K$ there can be at
most one bad parameter that happens to be in this case that corresponds to
$\psi^{\sigma}(t=\infty)$.
\item If $\deg(\gcd(\psi(t_k)\cdot d_i^{\sigma}(s) - n_i^{\sigma}(s), 1\leq
i\leq m))>1$ then $t_k$ is a bad parameter, since $\psi(t_k)$ is a singular
point.
\item If $\deg(\gcd(\psi(t_k)\cdot d_i^{\sigma}(s) - n_i^{\sigma}(s), 1\leq
i\leq m))=1$ but $\psi(t_k) = \psi^{\sigma}(\infty)$ then $t_k$ is a bad
parameter, $\psi(t_k)$ is singular.
\item If $\deg(\gcd(\psi(t_k)\cdot d_i^{\sigma}(s) - n_i^{\sigma}(s), 1\leq
i\leq m))=1$ and $\psi(t_k)\neq \psi^{\sigma}(\infty)$ then $t_k$ is a good
parameter, we compute $s_k$ solving the linear equation in $s$ given by the
$\gcd$.
\end{itemize}
\end{remark}

Hence, if $\C$ is defined over $\K$, we can compute each $u_j$ by interpolation.
We will need at most $d^2-2d+n+1+3=d^2-2d+n+4$ parameters. In practice however,
we will almost always need only $3$ parameters. Note also that if we choose the
parameters in $\K$, then all computations needed to compute $u_{i}$ are done in
$\K(\alpha, \alpha_i)$, that is a extension of degree bounded by $n(n-1)$.

If $\C$ is not defined over $\K$ it can happen two things while trying to
compute $u_j$. With high probability, we may find two different parameters such
that $\gcd(\psi(t_k)\cdot d_i^{\sigma}(s) - n_i^{\sigma}(s), 1\leq i\leq m)=1$
and this is a certificate that the curve is not defined over $\K$. On the other
hand, we may succeed computing $u_j$. This may happen if $\C=\C^{\sigma_j}$ for
this specific $\sigma_j$ or if we have chosen three parameters $t_k$ such that
$\psi(t_k)\in \C\cap \C^{\sigma_j}$. So, if we have computed all the linear
fractions $u_j(t)$ but we want a certificate that $\C$ is defined over $\K$, we
only need to check that $\psi(t) = \psi^{\sigma_j}(u_j(t))$, $1\leq i\leq n$. In
the case that computing this composition may be expensive, we can try to check
the equality evaluating in several parameters $t$. $\psi(t)$ and
$\psi^{\sigma_j}(u_j(t))$ are rational functions of degree $d$, so if they agree
on $2d+1$ parameters where both parametrizations are defined, then
$\psi(t)=\psi^{\sigma_j}(y_j(t))$ and $\C=\C^{\sigma_j}$. But there are $d$
parameters where $\C$ is not defined and other $d$ where $\C^{\sigma_i}$ is not
defined. So, if we want a certificate that $\C=\C^{\sigma_j}$ by evaluation, we
will need to try at most $4d+1$ parameters in the worst case. So $(n-1)(4d+1)$
evaluations to check all conjugates.

\begin{example}\label{ex:bound_is_sharp}
Let us show that the bounds given can be easily proven to be sharp if we allow
the parameters to be in $\F$ and $d\geq n$. Let $\K(\alpha)$ be normal over $\K$
of degree $n$ and $\sigma_1, \ldots, \sigma_{n-1}$ be $\K$-automorphisms that
send $\alpha$ onto its conjugates. The common denominator of the parametrization
of the curve will be $g=(t+1)\ldots(t+d)$ so that the parameters $-1,\ldots, -d$
will fail in the algorithm. Let us write a component as $f(t) = (\alpha
t^d+a_{d-1}t^{d-1}+\ldots +a_1t+a_0)/g(t)$, where the $a_i$ are indeterminates.
Impose the conditions $f(i)=\sigma_i(\alpha)$. This is a linear system of
equations in the $a_i$ representing an interpolation problem. We have $n-1$
conditions and $d$ unknowns in the system and $n-1<d$. Hence, there are
infinitely many solutions to the system and we can take two generic solutions
$f_1(t), f_2(t)$. The curve $\psi(t)=(f_1(t), f_2(t))$ will fail to give a
correct answer for $t=-1, \ldots, -d$ due to the denominator and for $t=i$,
$i=1, \ldots, n-1$ because $\psi(t=\infty)=(\alpha, \alpha)$, so
$\psi^{\sigma}(t=\infty)=(\sigma(\alpha), \sigma(\alpha))= \psi(i)$. Finally, if
we have chosen $f_1, f_2$ generic, the only singularities of $\psi$ will be
simple nodes in the affine plane. Thus, there will be $(d-1)(d-2)$ parameters
that will yield to a singularity. 

For a specific example, take $\K=\mathbb{Q}$, $\alpha$ a primitive $5-th$ root
of unity, so that $n=4$. Let the degree be $d=4$. If we perform the construction
above, we get the relations in the coefficients of $f$:
\[a_0 = -6a_3 + (1440\alpha^3 + 1080\alpha^2 + 1044\alpha + 1920)\]
\[a_1 = 11a_3 + (-1740\alpha^3 - 1440\alpha^2 - 1380\alpha - 2700)\]
\[a_2 = -6a_3 + (420\alpha^3 + 360\alpha^2 + 335\alpha + 780)\]
If we compute $f_0$ and $f_1$ substituting $a_3$ by $0$ and $1$ respectively, we
get the parametrization $\phi(f_0, f_1)$ of a rational curve of degree $4$ with
three nodes, such that the nodes are attained by the roots of $t^6 +
(-420\alpha^3 + 60\alpha^2 - 90\alpha - 102)t^5 + (-59220\alpha^3 -
171720\alpha^2 + 85110\alpha - 214952)t^4 + (688980\alpha^3 + 1237740\alpha^2 -
450750\alpha + 1759626)t^3 + (-2309580\alpha^3 - 3135240\alpha^2 + 714450\alpha
- 4869077)t^2 + (2877600\alpha^3 + 3308280\alpha^2 - 391560\alpha + 5387628)t +
(-1197360\alpha^3 - 1231920\alpha^2 + 42840\alpha - 2063124)$.
\end{example}

Now, we show how to avoid in some cases some computations of $u_j$ using
conjugation.

\begin{proposition}\label{prop:conjugated_automorphism}
Assume that $\mathcal{C}$ is defined over $\mathbb{K}$. Let $\alpha_i\neq
\alpha_j$ be two conjugates of $\alpha$ over $\mathbb{K}$. Suppose that
$\alpha_i$, $\alpha_j$ are also conjugated over $\mathbb{K}(\alpha)$ and that
$\tau$ is a $\mathbb{K}(\alpha)$-automorphism of $\mathbb{F}$ such that
$\tau(\alpha_i) = \alpha_j$. Then $\tau(u_i) = u_j$.
\end{proposition}
\begin{proof}
All operations to compute $u_j$ are evaluating rational functions with
coefficients in $\mathbb{K}(\alpha, \alpha_i)$ at parameters in $\mathbb{K}$ (or
even $\mathbb{Z}$), compute $\gcd$ of univariate polynomials with coefficients
also in $\mathbb{K}(\alpha, \alpha_i)$ and solving a linear system of equations.
These operations commute with conjugation by $\sigma$. Thus, if $\tau$ is a
$\K(\alpha)$-automorphism such that $\tau(\alpha_i)=\alpha_j$, we can conjugate
by $\tau$ at every step of the method to compute $u_i$. Hence, $\tau(u_i)=u_j$.
\end{proof}

If the Galois group of $\overline{\mathbb{K}(\alpha)}$ over $\mathbb{K}$ is the
permutation group $S_n$, we will only need to compute one automorphism $u_i$
making computations in a number field of degree $n(n-1)$. On the other extreme,
if $\mathbb{K}\subseteq \mathbb{K}(\alpha)$ is normal, we will have to compute
$n-1$ different automorphisms $u_i$, but the computations will be in the smaller
field $\mathbb{K}(\alpha)$.

Now, we show how to avoid computing in the normal closure of $\K(\alpha)$ over
$\K$ to solve the linear system of equations \ref{eq:Vandermonde_standard}. This
system is given by a Vandermonde matrix, so we are dealing with an interpolation
problem. If the standard parametrization searched is $(\phi_0, \ldots,
\phi_{n-1})$. Then, the polynomial
\[F(x) = \phi_0 + \phi_1 x +\ldots + \phi_{n-1} x^{n-1} \in
\mathbb{K}(\alpha)(t)[x]\]
is the unique polynomial of degree at most $n-1$ such that $F(\alpha_i) =
u_i(t)$, $0\leq i\leq n-1$. $F$ can be computed by Lagrange interpolation
\[F(x) = \sum_{i=0}^{n-1}
\frac{(x-\alpha_0)\ldots(x-\alpha_{i-1})(x-\alpha_{i+1})\ldots(x-\alpha_{n+1})}{
(\alpha_i-\alpha_0)\ldots(\alpha_i-\alpha_{i-1})(\alpha_i-\alpha_{i+1}
)\ldots(\alpha_i-\alpha_{n-1})}u_i(t)\]
Let us take a look at each term:
\[\frac{(x-\alpha_0)\ldots(x-\alpha_{i-1})(x-\alpha_{i+1})\ldots(x-\alpha_{n-1})
}{ (\alpha_i-\alpha_0)\ldots(\alpha_i-\alpha_{i-1})(\alpha_i-\alpha_{i+1}
)\ldots(\alpha_i-\alpha_{n-1})}\]

The numerator is $M(x)/(x-\alpha_i)=m(\alpha_i,x)$, where $M(x)$ is the minimal
polynomial of $\alpha$ over $\mathbb{K}$ and the denominator is $m(\alpha_i,
\alpha_i)=M'(\alpha_i)$. For each conjugacy class $\{\alpha_{i_1},\ldots,
\alpha_{i_j}\}$ of roots of $M(x)$ over $\mathbb{K}(\alpha)$, we have that
\[\sum_{k=1}^j \frac{m(\alpha_{i_k},x)}{m(\alpha_{i_k}, \alpha_{i_k})}
u_{i_k}(t) = trace \frac{m(\alpha_{i_1},x)}{m(\alpha_{i_1}, \alpha_{i_1})}
u_{i_1}(t).\]
Where the trace is computed for the extension
$\mathbb{K}(\alpha,t,x)\subseteq \mathbb{K}(\alpha, t,x)(\alpha_i)$. Hence, we
need to compute only one term of the Laurent interpolation for each conjugacy
class of roots of $M(x)$ over $\mathbb{K}(\alpha)$. These conjugacy classes are
determined by the factorization of $M(x)$ in $\mathbb{K}(\alpha)[x]$.

\begin{remark}\label{rem:compute_trace}
To compute fast the trace of $v = \frac{m(\alpha_{i},x)}{m(\alpha_{i},
\alpha_{i})} u_{i} \in \mathbb{K}(\alpha, t,\alpha_i)[x]$, first, we can compute
the Newton sums $\sum_{k=1}^j \alpha_{i_j}^l, 1\leq l\leq n-1$ from the minimal
polynomial of $\alpha_{i_1}$ over $\K(\alpha)$. If the coefficients of $v$ are
polynomials in $t$, we compute easily the trace of $v$ computing the trace of
each coefficient of $v$. If the coefficients of $v$ are not polynomials in $t$,
we can write $v$ as $n/(t+b)$, $b\in \mathbb{K}(\alpha, \alpha_i)$, $n \in
\K(\alpha, \alpha_i)[t]$. This is due to the fact that the variable $t$ only
appears on the term $u_{i}$ and it is a linear fraction. Now, let $g(t)$ be the
minimal polynomial of $-b$ over $\mathbb{K}(\alpha)$ and $g_1(t) = g(t)/(t+b)
\in \mathbb{K}(\alpha, \alpha_i)[t]$. Then $v=n/(t+b) = (n\cdot g_1(t))/g(t)$
and $trace(v) = trace(n\cdot g_1)/g(t)$ can be easily computed.
\end{remark}

Thus, we can compute the polynomial $F$ (i.e. the standard parametrization)
computing $\gcd$ and traces and norms in some fields of the form
$\mathbb{K}(\alpha, \alpha_i)$. To sum up, our algorithm to compute the standard parametrization of $\mathcal{U}$ is the following.

\begin{algorithm}\label{alg:main}
Input: A curve $\C$ given by a proper parametrization $\psi(\alpha,t)$ with
coefficients in $\mathbb{K}(\alpha)$.

Output: Either \emph{$\C$ is not defined over $\K$} or $\phi$, the standard
parametrization of the hypercircle associated to $\psi(t)$.

\begin{enumerate}
\item Set $M(x)$ the minimal polynomial of $\alpha$ over $\K$.
\item Set $m(\alpha, x)=M(x)/(x-\alpha) \in \mathbb{K}(\alpha)[x]$.
\item Compute $m(\alpha,x) = f_1(x)\cdots f_r(x)$ the factorization of
$m(\alpha, x)$ over $\mathbb{K}(\alpha)$.
\item Set $F=\frac{m(\alpha, x)}{m(\alpha, \alpha)}t\in
\mathbb{K}(\alpha,t)[x]$.
\item For $1\leq i\leq r$ do
\begin{enumerate}
\item Set $\alpha_i$ a root of $f_i(x)$.
\item Set $\psi^{\sigma_i}(t) = \psi(\alpha_i, t)$ the parametrization of the
curve
$C^{\sigma_i}$.
\item Compute three good parameters $t_1, t_2, t_3$ in the sense of
remark \ref{rem:cases_of_bad_good_parameters}.
\item If two parameters $t_i$, $t_j$ are found such that $\psi(t_i)$ and
$\psi(t_j)$ are well defined but  not attained by $\psi^{\sigma_i}$ then
Return \emph{$\C$ is not defined over $\K$}.
\item Compute $s_k$ such that $\psi(t_k) = \psi^{\sigma_i}(s_k)$, $1\leq k\leq
3$.
\item Compute $u_i(t)=\frac{at+b}{ct+d}$ the linear fraction such that
$u(t_k)=s_k$.
\item If $\psi \neq \psi^{\sigma_i}(u_i)$ the Return \emph{$\C$ is not defined
over $\K$}.
\item Compute $v = m(\alpha_i, x)/m(\alpha_i, \alpha_i)\cdot u_i(t) \in
\mathbb{K}(\alpha, t, \alpha_i)[x]$.
\item Compute $w = trace(v)$ for the extension $\mathbb{K}(\alpha,t,x)\subseteq
\mathbb{K}(\alpha,t,x)(\alpha_i)$.
\item Set $F = F+w \in \mathbb{K}(\alpha,t)[x]$.
\end{enumerate}
\item Write $F = \phi_0(t) + \phi_1(t)x+\ldots +\phi_{n-1}(t)x^{n-1}$.
\item Return $\phi = (\phi_0, \ldots, \phi_{n-1})$.
\end{enumerate}
\end{algorithm}

\begin{example}\label{ex:full_worked}
Now we present a full small example of the algorithm. Let
$\mathbb{K}=\mathbb{Q}$, $\alpha$ a root of $M(x)=x^4-2$, consider the proper
parametrization $\psi$ of a plane curve:\\
$x=((11\alpha^3 + 15\alpha^2 + 9\alpha + 11)t^3 + (7\alpha^3 + 14\alpha^2 +
14\alpha + 7)t^2 + (\alpha^3 + 2\alpha^2 + 4\alpha + 1)t)/D$\\ $y= ((15\alpha^3
+ 9\alpha^2 + 11\alpha + 22)t^3 + (25\alpha^3 + 29\alpha^2 + 16\alpha + 25)t^2 +
(9\alpha^3 + 18\alpha^2 + 15\alpha + 9)t + \alpha^3 + 2\alpha^2 + 4\alpha +
1)/D$,\\ with $D=(7t^3 + (12\alpha^3 + 3\alpha^2 + 6\alpha + 12)t^2 + (6\alpha^3
+ 12\alpha^2 + 3\alpha + 6)t + \alpha^3 + 2\alpha^2 + 4\alpha + 1)$.

Now, $M(x)=(x - \alpha) (x + \alpha) (x^2 + \alpha^2)$ is the factorization of
$M(x)$ in $\mathbb{K}(\alpha)[x]$. $m(\alpha, x) = (x + \alpha) (x^2 +
\alpha^2)$ and $m(\alpha, \alpha)= 4\alpha^3 = M'(\alpha)$. start with
$F=\frac{m(\alpha,x)}{m(\alpha, \alpha)}t=1/8(\alpha x^3t + \alpha^2x^2t +
\alpha^3xt + 2t)$.

From the factors of $m(x)$ we have two conjugacy classes of roots of $m$ over
$\mathbb{K}(\alpha)$. The first one is $\{-\alpha\}$. Let $\sigma$ be a
$\mathbb{Q}$-automorphism such that $\sigma(\alpha)=-\alpha$. Hence, we consider
the conjugate parametrization $\psi^{\sigma}$:\\ $x= ((-11\alpha^3 + 15\alpha^2
- 9\alpha + 11)t^3 + (-7\alpha^3 + 14\alpha^2 - 14\alpha + 7)t^2 + (-\alpha^3 +
2\alpha^2 - 4\alpha + 1)t)/D_1$,\\ $y=((-15\alpha^3 + 9\alpha^2 - 11\alpha +
22)t^3 + (-25\alpha^3 + 29\alpha^2 - 16\alpha + 25)t^2 + (-9\alpha^3 +
18\alpha^2 - 15\alpha + 9)t - \alpha^3 + 2\alpha^2 - 4\alpha + 1)/D_1$, with
$D_1=(7t^3 + (-12\alpha^3 + 3\alpha^2 - 6\alpha + 12)t^2 + (-6\alpha^3 +
12\alpha^2 - 3\alpha + 6)t - \alpha^3 + 2\alpha^2 - 4\alpha + 1)$.

We have to compute the automorphism $u_\sigma$ such that $\psi(t) =
\psi^\sigma(u_\sigma(t))$. we evaluate $(\psi^{\sigma})^{-1}(\psi(t_k))$ and
obtain:
\[\psi(0) = \psi^{\sigma}(0)\]
\[\psi(1) = \psi^\sigma(8/31\alpha^3 - 4/31\alpha^2 + 2/31\alpha - 1/31)\]
\[\psi(2) = \psi^\sigma(128/511 \alpha^3 - 32/511\alpha^2 + 8/511\alpha -
2/511)\]
Hence, $u_\sigma = \frac{at+b}{ct+d}$ is such that $u_\sigma(0)=0$,
$u_\sigma(1)=8/31\alpha^3 - 4/31\alpha^2 + 2/31\alpha - 1/31$,
$u_\sigma(2)=128/511 \alpha^3 - 32/511 \alpha^2 + 8/511 \alpha - 2/511$. We can
compute $u(t) = \frac{at+b}{ct+d}$ by solving a linear homogeneous system of
equations and get the solution
\[u_\sigma(t) = \frac{\alpha^3t}{4t + \alpha^3}\]
In this case $\psi^{\sigma}(u_\sigma) = \psi$, so $C=C^\sigma$. We can update
$F$ by adding:
\[m(-\alpha, x)/m(-\alpha,-\alpha)u_\sigma(t)=\frac{-x^3t + \alpha x^2t
-\alpha^2xt + \alpha^3t}{16t + 4\alpha^3}\]
So now:
\[F=\frac{\alpha x^3t^2 + \alpha^2 x^2t^2 + \alpha x^2t + \alpha^3 xt^2 +
2t^2 + \alpha^3 t}{8t + 2\alpha^3}\]
For this root, all operations are done in $\mathbb{K}(\alpha)$ since
$\sigma(\alpha)=-\alpha\in \mathbb{K}(\alpha)$.

Now, we have to deal with the roots of $x^2+\alpha^2$. Let $\beta$ be a root of
$x^2+\alpha^2$ and $\tau$ a $\mathbb{Q}$-automorphism such that
$\tau(\alpha)=\beta$. Consider the conjugate parametrization $\psi^\tau$:
$x=(((-11\alpha^2 + 9)\beta - 15\alpha^2 + 11)t^3 + ((-7\alpha^2 + 14)\beta -
14\alpha^2 + 7)t^2 + ((-\alpha^2 + 4)\beta - 2\alpha^2 + 1)t)/D_2$, $y =
(((-15\alpha^2 + 11)\beta - 9\alpha^2 + 22)t^3 + ((-25\alpha^2 + 16)\beta -
29\alpha^2 + 25)t^2 + ((-9\alpha^2 + 15)\beta - 18\alpha^2 + 9)t + (-\alpha^2 +
4)\beta - 2\alpha^2 + 1)/D_2$, where $D_2=7t^3 + ((-12\alpha^2 + 6)\beta -
3\alpha^2 + 12)t^2 + ((-6\alpha^2 + 3)\beta - 12\alpha^2 + 6)t + (-\alpha^2 +
4)\beta - 2\alpha^2 + 1$. In this case, we are taking the relative base
$\{\alpha^i\beta^j\ |\ 0\leq i \leq 3, 0\leq j\leq 1\}$ of $\mathbb{Q}(\alpha,
\beta)$ over $\mathbb{Q}$. Now we compute $u_\tau$ such that
$\psi(t)=\psi^\tau(u(t))$. for this
\[\psi(0)=\psi^\tau(0)\]
\[\psi(1) = \psi^\tau((2/9\alpha^2 - 2/9\alpha + 1/9)\beta + 2/9\alpha^3 -
1/9\alpha + 1/9)\]
\[\psi(2) = \psi^\tau((32/129\alpha^2 - 16/129\alpha + 4/129)\beta +
32/129\alpha^3 - 4/129\alpha + 2/129)\]
From this data, we can compute:
\[u_\tau(t) = \frac{t}{(\alpha-\beta)t + 1}\]
If $\gamma$ is the other root of $x^2+\alpha^2$ ( i.e. $\gamma=-\beta$) and
$\delta$ is a $\mathbb{Q}$-automorphism such that $\delta(\alpha)=\gamma$, then
$u_\gamma(t)=t/(\alpha-\gamma)t+1$. We have to compute the trace of 
\[v=\frac{m(\beta, x)}{m(\beta, \beta)}u_\tau(t)=\frac{\beta x^3t -\alpha^2x^2 t
-\alpha^2\beta xt + 2t}{(-8\beta + 8\alpha)t + 8}\]
over $\mathbb{Q}(\alpha, t)$. This is done using the technique described in
Remark~\ref{rem:compute_trace}.
\[trace(v)=\frac{-\alpha^2x^3t^2 -\alpha^3x^2t^2 -\alpha^2x^2t + 2xt^2 + 2\alpha
t^2 + 2t}{8\alpha^2t^2 + 8\alpha t + 4}\]
To compute this part, we have made computation in $\mathbb{Q}(\alpha, \beta)$.
We add $trace(v)$ to $F$ and get
\[F=\phi_0 +\phi_1 x +\phi_2x^2 +\phi_3x^3\]
where
\[\phi_0=\frac{2t^4 + 3\alpha ^3t^3 + 3\alpha ^2t^2 + \alpha t}{8t^3 + 6\alpha
^3t^2 + 4\alpha ^2t + \alpha }, \phi_1=\frac{\alpha ^3t^4 + 2\alpha ^2t^3 +
\alpha t^2}{8t^3 + 6\alpha ^3t^2 + 4\alpha ^2t + \alpha },\]
\[\phi_2=\frac{\alpha ^2t^4 + \alpha t^3}{8t^3 + 6\alpha ^3t^2 + 4\alpha ^2t +
\alpha }, \phi_3=\frac{\alpha t^4}{8t^3 + 6\alpha ^3t^2 + 4\alpha ^2t + \alpha
}\]
And $\phi=(\phi_0, \phi_1, \phi_2, \phi_3)$ is the standard parametrization of
the hypercircle associated to $\psi$.
\end{example}

So far, Algorithm~\ref{alg:main} only computes the hypercircle $\mathcal{U}$.
The algorithm is able to detect if $\C$ is not defined over $\K$, but apart
from that it does not provide much more useful information. In the rest of the
section, we show that, if $\C$ is not defined over $\K$, how can we compute the
minimum field $\mathbb{L}$ such that $\mathbb{K}\subseteq \mathbb{L} \subseteq
\K(\alpha)$ and $\C$ is defined over $\mathbb{L}$. Note that $\K(\alpha)$
always is a field of definition of $\C$, so the existence of $\mathbb{L}$ is
always guaranteed.

\begin{theorem}
Let $\C$ be a curve not $\K$-definable but $\K(\alpha)$-parametrizable. Let
$\mathbb{L}$ be the minimum field of definition of $\C$ containing $\K$.
$\K\subseteq \mathbb{L} \subseteq \K(\alpha)$. Then $\mathbb{L}$ is the subfield
of the normal closure $\overline{\K(\alpha)}$ over $\K$ that is fixed by the
$\K$-automorphisms $\sigma$ of $\overline{\K(\alpha)}$ such that $\C=\C^\sigma$.
\end{theorem}
\begin{proof}
First, we recall that the intersection of fields of definition of $\C$ is a
field of definition of $\C$. Hence, since $\K(\alpha)$ is a field of definition,
there always exists a minimum field of definition $\mathbb{L}$ of $\C$
containing $\K$.

From \cite{Ultraquadrics} it follows that if $\mathbb{L}_1 \subseteq
\mathbb{L}_2$ is any algebraic finite normal extension and $\mathbb{L}_2$ is a
field of definition of $\C$, then $\mathbb{L}_1$ is a field of definition of
$\C$ if and only if $\C=\C^\sigma$ for all $\sigma \in Aut(\mathbb{L}_2 /
\mathbb{L}_1)$.

Let $G= \{\sigma \in Aut(\overline{\K(\alpha)}/\K)\ |\ \C^\sigma =\C \}$.
Clearly, $G$ is a subgroup of $Aut(\overline{\K(\alpha)}/\K)$. This follows from
the fact that $(C^\sigma)^\tau = C^{\tau\circ \sigma}$. Let $\mathbb{L}$ be the
subfield of $\overline{\K(\alpha)}$ that is fixed by $G$. $\mathbb{L} \subseteq
\overline{\K(\alpha)}$ is a normal extension and, if $\sigma$ is a
$\mathbb{L}$-automorphism of $\K(\alpha)$ then $\sigma \in G$ so
$\C=\C^{\sigma}$. In this conditions, $\mathbb{L}$ is a field of definition of
$\C$. Moreover, it is the smallest field of definition of $\C$ containing $\K$.
If $\mathbb{K} \subseteq \mathbb{L}_1\subsetneq \mathbb{L}$ is a subfield of
$\mathbb{L}$, then $G_1$, the set of $\mathbb{L}_1$-automorphisms of
$\overline{K(\alpha)}$, is $G_1\supsetneq G$. Hence, there is an automorphism
$\tau \in G_1\setminus G$. But then $\C\neq \C^\tau$ and $\mathbb{L}_1$ cannot
be a field of definition of $\C$. Now, since $\K(\alpha)$ is also a field of
definition of $\C$, then $\mathbb{L}\subseteq \mathbb{K}(\alpha)$.
\end{proof}

If $\sigma_0=Id, \sigma_1, \ldots, \sigma_{n-1}$ are the automorphisms defined
in Section~\ref{sec:syntetic_construction}, then for any $\sigma\in
Aut(\K(\alpha)/\K)$, it happens that $\C^\sigma = \C^{\sigma_i}$ for some $i$,
$0\leq i\leq n-1$. Hence
\[\mathbb{L} = \bigcap_{\substack{0\leq i\leq n-1 \\ C=C^{\sigma_i}}} \{x\in
\K(\alpha)\ |\ \sigma_i(x)=x\}\]

If $\C$ is not defined over $\K$, we compute in step 5 of
Algorithm~\ref{alg:main} the set of automorphism $\sigma_i$ such that
$\C=\C^{\sigma_i}$. For any such $i$, let $m$ be the degree of $\alpha_i$ over
$\K(\alpha)$. If $x\in \mathbb{K}(\alpha)$, we can write $\sigma_i(x) =
\sum_{j=0}^{m-1} = l_i \alpha_i^j$, where $l_i\in \K(\alpha)$. $x$ is $\sigma_i$
invariant if and only if $x=l_0$, $l_i=0$, $1\leq i\leq m-1$. This provide a set
of $\K$-linear equations in the coordinates of $x$ in $\K(\alpha)\equiv \K^n$.
Note also that if $\alpha_i$ and $\alpha_j$ are conjugate over $\K(\alpha)$, the
equations imposed by $\sigma_i$ and $\sigma_j$ are the same. Hence, we only need
to compute them once for each set of conjugate roots of $M(x)$ over
$\K(\alpha)$. Solving the system of linear equations provide a base of
$\mathbb{L}$ as a $\K$-subspace of $\K(\alpha)$. From this equation, we may
reapply Algorithm~\ref{alg:main} but to the extension $\mathbb{L} \subseteq
\mathbb{K}(\alpha)$. In this case we already have computed the automorphisms
$u_i$ so we can reuse this computation.

\section{Complexity and Running Time}\label{sec:running}

We now compute the complexity of Algorithm~\ref{alg:main} in terms of number of
operations over the ground field $\mathbb{K}$. The analysis is by no means
sharp, we only intend to prove that there is a polynomial bound and that the
main obstacle is the degree of $\alpha$ over $\mathbb{K}$.

\begin{theorem}
Let $\mathbb{K}$ be a computable field with factorization of characteristic
zero. $\alpha$ algebraic of degree $n$ over $\mathbb{K}$ of minimal polynomial
$M(x)$. Let $\psi(t) = (\psi_0, \ldots, \psi_{m-1})$ be a proper parametrization
of a spatial curve $\C$ with coefficients in $\mathbb{K}(\alpha)$. Then the
number of operations over $\mathbb{K}$ of Algorithm~\ref{alg:main} is bounded by
$K + \mathcal{O}(md^5n^8)$ where $K$ is the time needed to factor $M(x)$ in
$\mathbb{K}(\alpha)[x]$.
\end{theorem}
\begin{proof}
We only use naive algorithms. The factorization of $M[x]$ can be performed
standard methods \cite{Cohen, Trager-trick} from a factorization algorithm in
$\mathbb{K}[x]$. Addition in $\mathbb{K}(\alpha)$ costs $n$ operations and
multiplication costs $\mathcal{O}(n^2)$ operations and inversion
$\mathcal{O}(n^3)$. If $\beta$ is a conjugate of $\alpha$, the worst case
complexity of addition in $\mathbb{K}(\alpha, \beta)$ is $\mathcal{O}(n^2)$
while multiplication is $\mathcal{O}(n^4)$ and inversion $\mathcal{O}(n^6)$. If
$f$ and $g$ are two polynomials of degree at most $d$, their $\gcd$ costs
$\mathcal{O}(d^3n^2+n^3d^2)$ operations in $\mathbb{K}(\alpha)$ or
$\mathcal{O}(d^3n^4+n^6d^2)$ if their coefficients live in $\mathbb{K}(\alpha,
\beta)$. Steps $1-3$ of the algorithm cost $K+\mathcal{O}(n^2)$. Step $4$ is
evaluating a polynomial in $\mathbb{K}(\alpha)$, invert the result and multiply
the polynomial this result. By Horner's method it is $\mathcal{O}(n^3)$. Step
$5.b$ can be done in $\mathcal{O}(dmn)$ operations. For a parameter $t_k$ doing
steps $5.c-d$ is evaluating $m$ rational functions in $\mathbb{K}(\alpha)$ and
then compute $m-1$ $\gcd$ in $\mathbb{K}(\alpha, \beta)$ of degree $d$, this
costs $\mathcal{O}(md^3n^4+mn^6d^2)$. From
Theorem~\ref{teo:number_of_univariate_gcd_for_hypercircle} we have to try at
most $\mathcal{O}(d^2+n)$ times, so the total cost is bounded by
$\mathcal{O}(md^5n^7)$.

Computing step $5.f$ is just solving a system of $3$ linear equations in $4$
unknowns in $\mathbb{K}(\alpha, \beta)$. This can be done in $\mathcal{O}(n^4)$
operations. Now, comparing $\psi$ and $\psi^\sigma(u)$ in $5.e$ can be done
evaluating both functions in $\mathcal{O}(d)$ parameters. Each evaluation costs
$\mathcal{O}(n^6)$, so in total, this step can be done in $\mathcal{O}(mdn^6)$.
Step $5.h$ we already have $m(\alpha_i,x)/m(\alpha_i, \alpha_i)$ precomputed by
conjugation, so we only need to multiply the polynomials, which is dominated by
computing $\mathcal{O}(n)$ products ($u$ is always of degree $\leq 1$ and we do
not need to do anything with the denominator). This costs $\mathcal{O}(n^5)$.
Now, instead of computing the minimal polynomial of the pole of $u$, we can
compute its characteristic polynomial over $\mathbb{K}(\alpha)$. Since the
characteristic polynomial of an $n\times n$ matrix can be done in
$\mathcal{O}(n^4)$ operations and the matrix will have entries in
$\mathbb{K}(\alpha)$, we can compute this characteristic polynomial in
$\mathcal{O}(n^6)$ operations. Step $5.j$ can be done in $O(n^6)$ operations.
Hence step $5$ is bounded by $\mathcal{O}(md^5n^7)$. Since we have to perform
step $5$ at most $n$ times. We get a bound of $\mathcal{O}(md^5n^8)$ operations
over $\mathbb{K}$.

If $\C$ is not $\K$-definable. In step $5$ we compute the automorphisms
$\sigma_i$ such that $\C = \C^\sigma$. From this automorphisms, we can compute
the field of definition $\mathbb{L}$ in $\mathcal{O}(n^4)$ operations and repeat
the whole algorithm. It is clear that the running time for the extension
$\mathbb{L} \subseteq \mathbb{K}(\alpha)$ is bounded by the case $\mathbb{K}
\subseteq \mathbb{K}(\alpha)$. So the global bound does not change.
\end{proof}

This result agrees with experimentation, the most important parameter is the
degree of $\alpha$ over $\K$ and the ambient dimension of $\mathcal{C}$ tend to
be not relevant in the algorithm compared to the other parameters.

Computing the hypercircle using Definition~\ref{defn:witness} is too slow,
because we have to work with an ideal in $n$ variables over $\K$ and make the
quotient by the ideal defined by the denominator. In \cite{SV-Quasipolynomial}
the authors proposed a method to compute the parametrization of the hypercircle.
It is based in the following result.

\begin{theorem}\label{teo:Sendra-Villarino}
Let $\psi$ be a proper parametrization of $\C$ with coefficients in
$\mathbb{K}(\alpha)$. Let $G(x_1, \ldots, x_m):\C\rightarrow \F$ be the inverse
of the parametrization. $G\in \mathbb{K}(\alpha)(x_1, \ldots, x_m)$. Write
$G=\sum_{i=0}^{n-1} G_i\alpha^i$, $G_i\in \mathbb{K}(x_1, \ldots, x_m)$, $0\leq
i\leq n-1$. Consider $\phi=(G_0(\psi), \ldots, G_{n-1}(\psi))$. Then $\C$ is
defined over $\mathbb{K}$ if and only if $\phi$ is well defined and parametrizes
a curve in $\mathbb{F}^n$. In this case $\phi$ is the standard parametrization
of the associated hypercircle to $\psi$. 
\end{theorem}
\begin{proof}
See \cite{SV-Quasipolynomial}
\end{proof}

Algorithm~\ref{alg:main} and the algorithm in Theorem~\ref{teo:Sendra-Villarino}
have been implemented in the Sage CAS \cite{sage-4-6-2}, the code for the method
presented in this paper can be obtained from \cite{sage-hypercircles}. We are
interested in the average case, so we will assume that our curve is planar
(since we can always make a generic projection). However, the method presented
here also performs well for spatial curves. For the method based on the inverse
of the parametrization of \cite{SV-Quasipolynomial} we compute $(G_0, \ldots,
G_{n-1})$ but we do not simplify the composition $G_i(\psi)$. This is done to
avoid artifacts in the running time that appeared if we simplify the
composition. The inverse of $\psi$ is computed using the resultant method
explained in \cite{MR2126926}.

We show the results for random curves of degree $2, 5, 10, 25$ and $50$. First
over an extension of $\mathbb{Q}$ of degree $2$, over a cyclotomic extension of
degree $6$ and a random extension of degree $5$. In all these cases $\C$ is
defined over $\K$.

\begin{center}
\begin{tabular}{|c|r|r|r|r|r|}
\hline
\multicolumn{6}{|c|}{Case: $\alpha^2+1=0$}\\
\hline
method $\backslash$ degree of $\C$ & 2 & 5 & 10 & 25 & 50\\
\hline
Moving hyperplanes: & 0.08 & 0.15 & 0.27 & 0.87 & 2.58\\
Inverse-based method: & 0.03 & 0.15 & 13.16 & $>60$ &\\
\hline
\multicolumn{6}{c}{}\\
\hline
\multicolumn{6}{|c|}{Case:
$\alpha^6+\alpha^5+\alpha^4+\alpha^3+\alpha^2+\alpha+1=0$}\\
\hline
Moving hyperplanes: & 1.12 & 2.00 & 3.71 & 13.13 & 48.01\\
Inverse-based method: & 0.13 & 28.61 & $>60$ & &\\
\hline
\multicolumn{6}{c}{}\\
\hline
\multicolumn{6}{|c|}{Case: $\alpha$ of degree 5, random minimal polynomial}\\
\hline
Moving hyperplanes: & 1.04 & 2.01 & 3.84 & 14.28 & $>60$\\
Inverse-based method: & 0.12 & 10.36 & $>60$ & &\\
\hline
\end{tabular}
\end{center}

Now, we show a table with a random extension of degree $5$ but $\C$ is not
defined over $\K$. In this case is more evident that the moving hyperplanes
method is better. With high probability it will detect that the curve is not
defined over $\K$ while trying to compute the automorphisms $u(t)$, on the
other hand, the inverse-based method always has to compute the inverse of the
parametrization $\psi$. In all cases, our algorithm computed the
minimum field of definition of the corresponding curve.

\begin{center}
\begin{tabular}{|c|r|r|r|r|r|}
\hline
\multicolumn{6}{|c|}{Case: $\alpha$ of degree 5, $\C$ not defined over $\K$}\\
\hline
method $\backslash$ degree of $\C$ & 2 & 5 & 10 & 25 & 50\\
\hline
Moving hyperplanes: & 0.36 & 0.59 & 0.91 & 2.16 & 5.13\\
Inverse-based method: & 0.08 & 12.68 & $>60$& & \\
\hline
\end{tabular}
\end{center}

\section*{Acknowledgements}
The author is supported by the Spanish ``Ministerio de Ciencia e Innovaci\'on" project MTM2008-04699-C03-03 and the work was completed during a stay of the author in the Mittag-Leffler 2011 spring program.

%
%
\end{document}